\documentclass[reqno, 12pt]{amsart}
\usepackage{amsmath,amssymb,amsthm}    
\usepackage{mathabx}    
\usepackage{mathrsfs}    	
\usepackage[backref,colorlinks]{hyperref} 
\usepackage[abbrev,msc-links]{amsrefs}  
\usepackage{enumerate}
\usepackage{enumitem}

\normalsize                              
\newtheorem{theorem}{Theorem}
\newtheorem{definition}[theorem]{Definition}

\newtheorem{claim}[theorem]{Claim}

\newtheorem{note}[theorem]{Note}
\newtheorem{remark}[theorem]{Remark}

\def\dHH#1{\leavevmode\setbox0=\hbox{#1}\dimen0=\wd0\setbox0=\hbox{.}%
	\advance\dimen0 by -\wd0%
	\hbox{#1\raise-0.5ex\hbox to 0pt{\hss.\kern.5\dimen0}}}%

\newcommand{\xx}{\mathcal{X}}
\newcommand{\yy}{\mathcal{Y}}
\newcommand{\zz}{\mathcal{Z}}
\renewcommand{\aa}{\mathcal{A}}
\newcommand{\bb}{\mathcal{B}}
\newcommand{\cc}{\mathcal{C}}

\newcommand{\lr}{\Leftrightarrow}

\begin{document}
\title[Note on Ramsey theorem for posets]{Note on Ramsey theorem for posets with linear extensions}

\author[A.Arman]{Andrii Arman}
\address{Department of mathematics, University of Manitoba, Winnipeg, Manitoba R3T 2N2, Canada}
\email{andrew0arman@gmail.com}
\thanks{}

\author[Vojt\v{e}ch R\"{o}dl]{Vojt\v{e}ch R\"{o}dl}
\address{Department of Mathematics and Computer Science, 
Emory University, Atlanta, GA 30322, USA}
\email{rodl@mathcs.emory.edu}
\thanks{The second  author was supported by NSF grant DMS 1301698}


\date{\today}


\begin{abstract} 
In this note we consider a Ramsey type result for partially ordered sets. In particular we give an alternative short proof of a theorem for a posets with multiple linear extensions recently obtained by Solecki and Zhao in \cite{SZ}.
\end{abstract}

\maketitle

\section{Preliminary definitions}
A poset is a pair $(X,P^{X})$, where $X$ is a set and  $P^{X}$ is a partial order on $X$. We consider partial orders that are strict, i.e. not reflexive. 

We say that a partial order $L^{X}$ on $X$ \textit{extends} a partial order $P^{X}$ on $X$ if for all $x, y \in X$
$$x P^{X} y \Rightarrow x L^{X} y.$$ 

If $(X,P^{X})$ is a poset and $U \subset X$ we denote by $P^{X}|_{U}$ the \textit{restriction} of $P^{X}$ onto $U$. 

Below, we consider collections  $\mathcal{L}^{X}_k=(L_1^{X}, L_2^{X}, \dots, L_k^{X})$ ,where each of $L_{i}^{X}$ is a linear order on $X$.

\begin{definition}
We denote by $PL^{(k)}$ the set consisting of all
triplets $(X, P^{X}, \mathcal{L}^{X}_k)$, where $(X,P^{X})$ is a poset and each $L_{i}^{X}$ for $ i \in [k]$ is a linear order that extends $P^{X}$.
\end{definition}

\begin{definition}
Let $\xx , \yy \in PL^{(k)}$, where $\xx=(X, P^{X}, \mathcal{L}^{X}_k)$ and $\yy=(Y, P^{Y}, \mathcal{L}^{Y}_k)$.
We  write $\xx \subseteq \yy$ if 
\begin{itemize}
\item $X\subseteq Y$ and $P^{Y}|_X$ extends $P^{X}$.
\item $L^{Y}_{i}|_X=L^{X}_{i}$ for all $i\in [k]$.
\end{itemize} 
\end{definition}

\begin{definition}

Let $\xx , \yy \in PL^{(k)}$, where $\xx=(X, P^{X}, \mathcal{L}^{X}_k)$ and $\yy=(Y, P^{Y}, \mathcal{L}^{Y}_k)$. We say that a mapping $\pi : X \to Y$ is order preserving for $\xx$ and $\yy$ if  for any $i \in [k]$ and any $x,y \in X$ we have 
$$ x L_{i}^X y \lr \pi(x)L_{i}^{Y}\pi(y) \; \; \;\text{and} \; \; \;  x P^{X}y \lr \pi(x)P^{Y}\pi(y) .$$
\end{definition}
\begin{definition} \label{def:isomor}
We say that $\pi$ is an isomorphism between $\xx \in PL^{(k)}$ and $\tilde{\xx} \in PL^{(k)}$ if it is order preserving bijection. We say that $\xx \in PL^{(k)}$ is isomorphic to $\tilde{\xx} \in PL^{(k)}$ if there is an isomorphism between $\xx$ and $\tilde{\xx}$.
\end{definition}

\begin{definition}\label{def:copy}
Let  $k > 0$ and  $\xx, \yy \in PL^{(k)}$. We say that $\tilde{\xx} \in PL^{(k)}$ is a copy of $\xx$ in $\yy$ if 
$\tilde{\xx} \subseteq \yy$ and $\tilde{\xx}$ is isomorphic to $\xx$.
For $\xx, \yy \in PL^{(k)}$ denote by  
$\binom{\yy}{\xx}$ the set of all copies of $\xx$ in $\yy$.

\end{definition}
For any $\tilde{\xx} \in \binom{\yy}{\xx}$ there is unique order preserving mapping $\pi: X \to \tilde{X}$. On other hand, any order preserving mapping $\pi : X \to Y$ induces a copy $\tilde{\xx}=\pi(\xx) \in \binom{\yy}{\xx}$.
We identify each $\tilde{\xx} \in \binom{\yy}{\xx}$ with corresponding order preserving mapping $\pi$ and will say that $\pi$ is a copy of $\xx$ in $\yy$ instead of saying that $\tilde{\xx}$ is a copy of $\xx$ in $\yy$ with corresponding order preserving mapping $\pi$. 

We refer to the following theorem as to Ramsey theorem for posets with one linear extension.

\begin{theorem}\label{thm:one_ext}
For any integer $r$ and any $\xx, \yy \in PL^{(1)} $ there is $\zz \in PL^{(1)}$, such that for any $r$-colouring of set $\binom{\zz}{\xx}$ there is $\tilde{\yy}$, a copy of $\yy$ in $\zz$, such that $\binom{\tilde{\yy}}{\xx}$ is monochromatic.  
\end{theorem}
Ramsey properties of the class of partially ordered sets were considered in \cite{NR} and \cite{PTW}, where all partially ordered sets with P-Ramsey properties were characterised (see also \cite{NR2}). Subsequently some extensions and related results were obtained in \cite{Promel} and \cite{Fouche}, using different method.  

Next theorem is a product version of the Theorem \ref{thm:one_ext}, that we are going to use in Section \ref{sec:proof}. Proof of this theorem is based on a standard folkloristic argument. For similar results of this type see e.g. \cite{Promel}.

\begin{theorem}\label{thm:product}
For any $\xx_i, \yy_{i}\in PL^{(1)}$ with $i\in [k]$  there are $\zz_i \in PL^{(1)}$ with $i \in [k]$, such that for any 2-colouring of set $\binom{\zz_1}{\xx_1} \times \dots \times \binom{\zz_k}{\xx_k}$  there are $\tilde{\yy_i}$, a copies of $\yy_i$ in $\zz_i$ for $i \in [k]$, such that $\binom{\tilde{\yy_1}}{\xx_1} \times \dots \times \binom{\tilde{\yy_k}}{\xx_k}$ is monochromatic.  
\end{theorem}

To distinguish between the objects of $PL^{(1)}$, which will play a special role in our proof, and  $PL^{(k)}$ for $k\geq 2$, from now on, we use letters $\xx$, $\yy$ and $\zz$ for elements of $PL^{(1)}$ and $\aa$, $\bb$, $\cc$ for elements of $PL^{(k)}$.

Based on Theorem \ref{thm:product}, in Section \ref{sec:proof} we are going to prove the following result, first obtained in \cite{SZ} .

\begin{theorem} \label{thm:main}
For any integer $k$ any $\mathcal{A}, \bb \in PL^{(k)} $ there is $\cc \in PL^{(k)}$, such that for any colouring $2$-colouring of set $\binom{\cc}{\mathcal{A}}$ there is $\tilde{\bb}$, a copy of $\bb$ in $\cc$, such that $\binom{\tilde{\bb}}{\mathcal{A}}$ is monochromatic.  
\end{theorem}

\section{Properties of join and canonical copies}
First, we define the join of $k$ elements of $PL^{(1)}$.

\begin{definition}\label{def:join}
Let $\zz_{i}=(Z_i, P^{Z_i}, L^{Z_i})\in PL^{(1)}$ for $i \in [k]$ ans set $C=\Pi_{i=1}^{k}Z_i$.

Define partial order $<_C$ on set $C$  by
$\overline{x}<_{C}\overline{y}$  if $x_{i}P^{Z_{i}}y_{i}$ for all $i \in [k]$.

For all $i \in [k]$ define shifted lexicographic orders $<_{lx_i}$ on set $\Pi_{i=1}^{k}Z_i$, by 

$$\overline{x}<_{lx_i}\overline{y} \lr x_{i+\delta}L^{Z_{i+\delta}}y_{i+\delta},$$
where $\delta$ is the smallest non-negative number $j$, for which $x_{i+j}\neq y_{i+j}$ (with addition mod $k$). Let $\mathcal{L}_{k}^{C}=(<_{lx_1}, <_{lx_2}, \dots, <_{lx_k})$.
Then the  join of $\zz_1, \dots, \zz_k$ is $$\cc=(C, P^C, \mathcal{L}_{k}^C).$$
For notation, we will use $\cc=\sqcup_{i=1}^k \zz_i$.
\end{definition}
Note, that for $\zz_1, \dots,  \zz_k \in PL^{(1)}$ we have that $\sqcup_{i=1}^{k} \zz_{2} \in PL^{(k)}$. Indeed, since  $L^{Z_{i}}$ extends $P^{Z_i}$ we infer that $<_{lx_i}$ also extends $<_{C}$ for $i \in [k]$.

\begin{claim}\label{claim:1}
Let $\zz_{i}=(Z_i, P^{Z_i}, L^{Z_i})\in PL^{(1)}$ for $i \in [k] $ and let $\aa=(X, P^{X}, \mathcal{L}^{X}_{k}) \in PL^{(k)}$. Set $\cc=\sqcup_{i=1}^{k} \zz_i$ and let $\pi_{i} : X \to Z_{i}$ be a copy of $\xx_{i}=(X, P^{X}, L_{i}^{X})$ in $\zz_{i}$ for $i \in [k]$. Then the image of the mapping $\pi : X \to C$ , defined by 
$$\pi(x)=(\pi_1(x), \pi_2(x), \dots, \pi_k(x))$$
for each $x \in X$, is a copy of $\mathcal{A}$ in $\binom{\cc}{\mathcal{A}}$.
\end{claim}
\begin{remark}\label{remark:1}$ $

\begin{itemize}

\item We say that the image of the mapping $\pi$ from Claim \ref{claim:1}, is a \textit{canonical} copy of $\mathcal{\bb}$ in $\cc=\sqcup_{i=1}^k\zz_k$.
\item By $\binom{\cc}{\mathcal{B}}_{can} \subseteq \binom{\cc}{\mathcal{B}}$ we denote a set of all canonical copies of $\mathcal{B}$ in $\cc$.
\end{itemize}

\end{remark}

\begin{proof}
We need to verify that $\pi: X \to C$ is order preserving for $\aa$ and $\cc$.  
Indeed, we observe that if $x,y \in X$, then fact that $\pi_i: X \to Z_{i}$ preserves $P^{X}$ for $i \in [k]$ combined with definition of $\cc$  yields
$$xP^Xy \lr 
\forall i\in [k] : \pi_i(x) P^{Z_i}\pi_i(y) 
\lr \pi(x)<_C\pi(y).$$
Since $\pi_i$ preserves $L^{X}_i$ for $i \in [k] \;$ , we have 
$$ xL_i^Xy \lr \pi_i(x)L_i^{Z_i}\pi_i(y) \lr \pi(x)<_{lx_{i}}\pi(y)$$
for $i \in [k]$. Hence, $\pi$ preserves $P^{X}$ and $L^{X}_i$ for $i \in [k]$.

\end{proof}
 For the rest of this section we assume that $\cc=\sqcup_{i=1}^{k}\zz_i=(C, <_C, \mathcal{L}^{C}_{k})$,  $\aa=(X, P^{X}, \mathcal{L}_k^{X})$ and $\bb=(Y, P^{Y}, \mathcal{L}_k^{Y})$.

\begin{note}\label{note:lambda}
By construction, $\binom{\cc}{\mathcal{A}}_{can}$ is in 1-1 correspondence with the set $\Pi_{i=1}^k\binom{\zz_i}{\xx_i}$ and the function $\lambda : (\pi_1(X), \dots , \pi_{k}(X)) \mapsto \pi(X)$ is the bijection between sets $\Pi_{i=1}^k\binom{\zz_i}{\xx_i}$ and $\binom{\cc}{\mathcal{A}}_{can}$.
\end{note}


The following Claim states that if $\pi$ is a canonical copy of $\bb$ in $\cc$ and $\tilde{\mathcal{A}}$ is a copy of $\mathcal{A}$ in $\bb$, then  $\pi(\tilde{\mathcal{A}})$  is a canonical copy of $\mathcal{A}$ in $\cc$.

\begin{claim}\label{claim:2}
If $\pi \in \binom{\cc}{\bb}_{can}$ and $\tau \in \binom{\bb}{\aa}$, then $\sigma=\pi \circ \tau \in \binom{\cc}{\aa}_{can}$.
\end{claim}
\begin{proof}
Since $\pi: Y \to C$ is a canonical copy, we have that $\pi=(\pi_1, \dots,  \pi_k),$ where $\pi_i : Y \to Z_i$ are copies of $Y$ in $Z_{i}$ for $i \in [k]$. Define $\sigma_{i}=\pi_{i}\circ \tau$ for $i=\in [k]$. It is sufficient to prove that for any $i \in [k]$  $\sigma_{i}$ is order preserving for $\xx$ and $\zz_{i}$.  

Indeed, since $\tau$ is order preserving for $\aa$ and $\bb$ and $\pi_{i}$ is order preserving for preserves $(Y, P^{Y}, L^{Y}_i)$ and $\zz_i$ for any  $i \in [k]$, we have for any $x,y \in X$ and for $i \in [k]$

$$xP^Xy \lr
 \tau(x)P^{Y}\tau(y) \lr  \pi_{i}(\tau(x))P^{Z_{i}}\pi_{i}(\tau(y))\lr \sigma_{i}(x)P^{Z_{i}}\sigma_{i}(y),$$

$$xL^X_{i}y \lr
 \tau(x)L^{Y}_{i}\tau(y) \lr \pi_{i}(\tau(x))L^{Z_{i}}\pi_{i}(\tau(y))\lr \sigma_{i}(x)L^{Z_{i}}\sigma_{i}(y).$$

Consequently, for $i \in [k]$, $\sigma_{i}$ is order preserving for $\xx$ and $\zz_{i}$, and $\sigma=(\sigma_1, \sigma_2)$ is a canonical copy of $\aa$ in $\cc$. 
\end{proof}

Our final Claim states that if $\tilde \bb$ is a canonical copy of $\bb$ in $\cc$, and $\tilde{\mathcal{A}}$ is a copy of $\mathcal{A}$ in $\tilde \bb$, then $\tilde{\mathcal{A}}$ is a canonical copy of $\mathcal{A}$ in $\cc$.

\begin{claim}\label{claim:3}
If $\pi \in \binom{\cc}{\bb}_{can}$ and $\sigma \in \binom{\pi(\bb)}{\aa}$, then $\sigma \in \binom{\cc}{\aa}_{can}$.
\end{claim}
\begin{proof}
Since $\pi$ is an isomorphism between $\bb$ and $\pi(\bb)$, then $\pi^{-1}$ exists and is order preserving for $\pi(\bb)$ and $\bb$. Therefore, $\tau=\pi^{-1} \circ \sigma$ is order preserving mapping for $\mathcal{A}$ and $\bb$. Finally, Claim \ref{claim:2} applied for $\pi$ and $\tau$ gives that $\pi \circ \tau =\sigma$ is canonical copy of $\mathcal{A}$.

\end{proof}

\section{Proof of Theorem \ref{thm:main}}\label{sec:proof}
Let $\mathcal{A}=(X, P^{X}, \mathcal{L}_{k}^{X})$ and $\bb=(Y,P^{Y}, \mathcal{L}_{k}^Y)$ be given. Applying Theorem \ref{thm:product} with $\xx_{i}=(X,P^{X}, L_{i}^X)$ for $i \in [k]$ and $\yy_i=(Y,P^{Y}, L_{i}^Y)$ for $i \in [k]$ we obtain $\zz_{i}=(Z_{i},P^{Z_{i}}, L_{i}^{Z_{i}})$ for $i \in [k]$ .

Set $\cc=\sqcup_{i=1}^{k} \zz_i$. Let $\chi: \binom{\cc}{\mathcal{A}} \mapsto \{red,blue\}$ be a colouring.
Since $\binom{\cc}{\mathcal{A}}_{can} \subseteq \binom{\cc}{\mathcal{A}}$, colouring $\chi$ induces $\{red,blue\}$ colouring of $\binom{\cc}{\mathcal{A}}_{can}$. By Note \ref{note:lambda}, sets in $\binom{\cc}{\mathcal{A}}_{can}$ and elements of $\Pi_{i=1}^{k}\binom{\zz_i}{\xx_i}$ are in 1-1 correspondence and thus $\lambda^{-1} \circ \chi$ induces a colouring of $\Pi_{i=1}^{k}\binom{\zz_i}{\xx_i}$. By a choice of $\zz_i$ (recall that $\zz_i \in PL^{(1)}, \; i \in [k]$) there are $\tilde{Y}_i \in \binom{\zz_{i}}{\yy_{i}}$ for $i \in [k]$, such that 
$\Pi_{i=1}^k\binom{\tilde{\yy}_i}{\xx_i}$ is monochromatic and w.l.o.g we assume that all elements of $\Pi_{i=1}^k\binom{\tilde{\yy}_i}{\xx_i}$ are red.

Let $\pi_{i} : \yy_i \to \tilde{\yy}_i$ be the corresponding isomorphism between $\yy_i$ and $\tilde{\yy}_i$ for $i \in [k]$ and
let $\pi : B \to C$ be a mapping defined by 
$\pi(y)=(\pi_1(y), \dots, \pi_k(y))$ for each $y \in Y$.
Then, by Claim \ref{claim:1}, $\tilde{\bb}=\pi(\bb)$ is a copy of $\bb$ in $\cc$. Let $\tilde{\mathcal{A}}$ be a copy of $\mathcal{A}$ in $\tilde{\bb}$, then, by Claim \ref{claim:3}, $\tilde{\mathcal{A}}$ is a canonical copy of $\mathcal{A}$ in $\cc$.

Let $\sigma$ be isomorphism from $\mathcal{A}$ to $\tilde{\mathcal{A}}$, then $\sigma=(\sigma_1, \dots, \sigma_k)$, where $\sigma_i$ is order preserving for $\xx_i$ and $\tilde{\yy}_i$ for any $i \in [k]$. Since $\tilde{\mathcal{A}}\in \binom{\cc}{\mathcal{A}}_{can}$ and all elements of $\Pi_{i=1}^k\binom{\tilde{\yy}_i}{\xx_i}$ are red, we get that $(\lambda^{-1} \circ \chi )(\tilde{\mathcal{A}})$ is red and  consequently $\tilde{\mathcal{A}}$ is red.
Therefore,
 set $\binom{\tilde{\bb}}{\mathcal{A}}$ is monochromatic.


\begin{bibdiv}
\begin{biblist}


\bib{Fouche}{article}{
   author={Fouch{\'e}, W. L.},
   title={Symmetry and the Ramsey degree of posets},
   note={15th British Combinatorial Conference (Stirling, 1995)},
   journal={Discrete Math.},
   volume={167/168},
   date={1997},
   pages={309--315},
   issn={0012-365X},
   review={\MR{1446753}},
   doi={10.1016/S0012-365X(96)00236-1},
}

\bib{NR}{article}{
   author={Ne{\v{s}}et{\v{r}}il, Jaroslav},
   author={R{\"o}dl, Vojt{\v{e}}ch},
   title={Combinatorial partitions of finite posets and lattices---Ramsey
   lattices},
   journal={Algebra Universalis},
   volume={19},
   date={1984},
   number={1},
   pages={106--119},
   issn={0002-5240},
   review={\MR{748915}},
   doi={10.1007/BF01191498},
}

\bib{NR2}{article}{
   author={N\vspace{0mm}e{\v{s}}et{\v{r}}il, Jaroslav},
   author={R{\"o}dl, Vojt{\v{e}}ch},
   title={Ramsey partial orders from acyclic graphs},
   note={arxiv.org:1608.04662},
   }

\bib{PTW}{article}{
   author={Paoli, M.},
   author={Trotter, W. T., Jr.},
   author={Walker, J. W.},
   title={Graphs and orders in Ramsey theory and in dimension theory},
   conference={
      title={Graphs and order},
      address={Banff, Alta.},
      date={1984},
   },
   book={
      series={NATO Adv. Sci. Inst. Ser. C Math. Phys. Sci.},
      volume={147},
      publisher={Reidel, Dordrecht},
   },
   date={1985},
   pages={351--394},
   review={\MR{818500}},
}


\bib{Promel}{book}{
   author={Pr{\"o}mel, Hans J{\"u}rgen},
   title={Ramsey theory for discrete structures},
   note={With a foreword by Angelika Steger},
   publisher={Springer, Cham},
   date={2013},
   pages={xvi+232},
   isbn={978-3-319-01314-5},
   isbn={978-3-319-01315-2},
   review={\MR{3157030}},
   doi={10.1007/978-3-319-01315-2},
}

\bib{SZ}{article}{
   author={Solecki, S.},
   author={Zhao, M.},
   title={A Ramsey theorem for partial orders with linear extensions},
   note={arXiv:1409.5846},
}

\end{biblist}
\end{bibdiv}
\end{document}